\documentclass[11pt]{article}
\usepackage{color,latexsym,amsthm,amsmath,amssymb,path,enumitem,url}
\usepackage{graphicx,wrapfig,lipsum}
\newcommand{\ignore}[1]{}
\usepackage [english]{babel}
\usepackage{hyperref}
\usepackage{float}

\newtheorem{theorem}{Theorem}[section]
\newtheorem{lemma}[theorem]{Lemma}

\newcommand{\Proof}[1]
        {
        \noindent
        \emph{Proof #1.}~
        }
\newsavebox{\smallProofsym}                     
\savebox{\smallProofsym}                        %
        {
        \begin{picture}(7,7)                    %
        \put(0,0){\framebox(6,6){}}             %
        \put(0,2){\framebox(4,4){}}             %
        \end{picture}                           %
        }                                       %
\newcommand{\smalleop}[1]
        {
        \mbox{} \hfill #1~~\usebox{\smallProofsym}\!\!\!\!\!\!\
        }

\newcommand{\parag}[1]{\vspace{2mm}

\noindent{\bf #1} }

\usepackage{graphicx,psfrag}
\usepackage[rflt]{floatflt}

\usepackage{dsfont}

\newcommand{\pts}{\mathcal P}

\newcommand{\rc}{\mathsf{rc}}
\newcommand{\pg}{\mathsf{pg}}
\begin{document}
 
\newpage
\title{An Upper Bound for the Number of Rectangulations of a Planar Point Set\thanks{Some of the research work of this project was done as part of the 2021 New York Discrete Math REU, funded by NSF grant DMS-2051026}}

\author{
Hannah Ashbach\thanks{Mount Mary University, Milwaukee, WI 53222, USA. {\sl ashbachh@gmail.com}}
\and
Kiki Pichini\thanks{Columbia University, New York, NY, 10027, USA. {\sl k.pichini@columbia.edu}} }
\maketitle
\begin{abstract}
We prove that every set of $n$ points in the plane has at most $(16+5/6)^n$ rectangulations. This improves upon a long-standing bound of Ackerman. Our proof is based on the cross-graph charging-scheme technique.
\end{abstract}

\section{Introduction}
Let $\pts$ be a finite point set in the plane. How many graphs can be embedded on $\pts$ with non-crossing straight edges? This natural question has a long and rich history. Tutte \cite[Chapter 10]{tutte} studied the problem for the case of unlabeled vertices. Ajtai, Chv\'atal, Newborn, and Szemer\'edi \cite{acns} introduced the ubiquitous crossing lemma to bound the maximum number of labelled graphs that can be embedded over any set of $n$ points. More precisely, denote by $\pg(\pts)$ the number of labelled plane graphs that can be embedded over $\pts$. By labelled, we mean that different embeddings of the same graph are counted as separate. The main goal of Ajtai, Chv\'atal, Newborn, and Szemer\'edi was to study $\pg(n)$, the maximum of $\pg(\pts)$ taken over every set $\pts$ of $n$ points. 

The above problem has many variants, some of which might be considered more interesting than the original. For example, one may wish to find the maximum number of Hamiltonian cycles that can be embedded over a set of $n$ points. Other main variants involve triangulations, spanning trees, and more (for example, see \cite{tri4, tri5, tri2, tri1, tri3,cgcs1}). Euler \cite{euler} introduced the famous Catalan numbers to study the number of triangulations of a point set in convex position. Beyond the extremal problems, algorithms for counting and enumerating such plane graphs are also being developed (for example, see \cite{alg1, alg3, alg2}).

We say that a set of points is in \emph{general position} if no two points share the same $x$- or $y$-coordinate. For a set $\pts$ of $n$ points in general position within a rectangle $B$, a \textit{rectangulation} $G$ of $(B, \pts)$ is a partition of $B$ into rectangles using axis-parallel segments, so that every segment contains a single point of $\pts$. The segments do not intersect in their interiors, although an endpoint of one segment may lie in the interior of another. For example, Figure \ref{a} depicts two rectangulations of the same point set $\pts$. Note that there is a bijection between the points of $\pts$ and the
segments used to partition $B$. We define $R(\pts)$ as the set of rectangulations of $\pts$, and $\rc(\pts)$ as the number of rectangulations of $\pts$. In other words, $\rc(\pts)=|R(\pts)|$. We denote by $\rc(n)$ the maximum of $\rc(\pts)$ taken over every set $\pts$ of $n$ points in general position.

\begin{figure}[h]
    \centering
    \includegraphics[scale=.76]{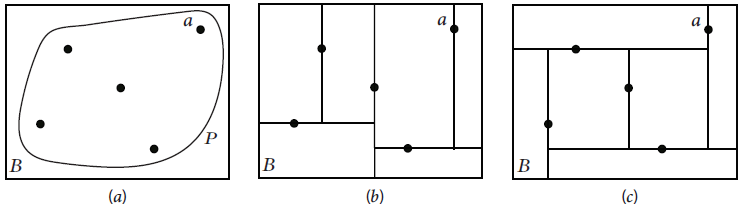}
    \caption{\small\sf (a) Point set $\pts$ in rectangle $B$. (b) The rectangulation $G$ of $(B, \pts)$. (c) The rectangulation $G'$ of $(B, \pts)$.}
    \label{a}
\end{figure}

Many combinatorial and algorithmic aspects of rectangulations have been studied (for example, see \cite{rec1, rec2, rec3}). Ackerman, Barequet, and Pinter \cite{rec1} proved $\rc(n)= O(20^n)$, which was later improved by Ackerman \cite{ack} to $\rc(n)=O(18^n\cdot n^4)$. Felsner \cite{felsner} then proved that $\rc(n)=\Omega(8^n\cdot n^4)$, and this is the current best lower bound. 
Referring to the upper bound for $\rc(n)$, Felsner also states that ``To improve this bound remains an intriguing problem.'' The current work further improves Ackerman's upper bound:

\begin{theorem}\label{mainthm}
$\rc(n)\le (16+5/6)^n$.
\end{theorem}

Our proof is based on the \textit{cross-graph charging-scheme} technique. This technique was introduced by Sharir and Welzl \cite{cgcs1} and further developed by Sharir and Sheffer \cite{cgcs2}. We show how such an approach could be pushed further to obtain stronger bounds for $\rc(n)$.

In Section \ref{sec2}, we introduce additional notation and properties of rectangulations. Section \ref{sec3} contains the proof of Theorem \ref{mainthm}.

\section{Rectangulation Preliminaries}\label{sec2}

In this section, we introduce several properties and definitions involving rectangulations. We will rely on those in our proof in Section 3.

Given a vertex \textit{a} of a rectangulation \textit{G}, we denote as $(a, G)$ the segment in $G$ containing $a$. We consider $(a, G)$ and $(a, G')$ as distinct segments, even if they are geometrically identical. For example, in Figures 1(b) and 1(c), the segments $(a, G)$ and $(a, G')$ are geometrically identical, but we consider these as distinct segments. An \textit{intersection} is defined as an endpoint of one segment lying in the interior of another segment or on $B$. We define the \textit{degree} of a segment $(a, G)$ as the number of intersections on $(a, G)$. For instance, in Figures 1(b) and 1(c), the segments $(a, G)$ and $(a, G')$ have degrees of 2 and 3, respectively. Note that no segment has degree smaller than 2, because the two endpoints of each segment are intersections by definition.

For an integer $j\ge 2$ and a rectangulation $G$, we denote by $d_j(G)$ the number of segments of degree $j$ in $G$. For example, with respect to the rectangulations in Figures 1(b) and 1(c), we have that $d_2(G)=2$ and $d_2(G')=1$. Finally, the expected value of $d_j(G)$ for a rectangulation chosen uniformly from $R(\pts)$ is denoted as $\hat{d}_j(\pts).$ In other words,
\begin{equation}
    \hat{d}_j(\pts)=\frac{\sum_{R\in R(\pts)}d_j(R)}{\rc(\pts)}.
\end{equation}

Consider segments $s_1=(a_1,G)$ and $s_2=(a_2,G)$ such that an endpoint $t$ of $s_1$ is contained in the interior of $s_2$. We define \textit{rotating} the intersection $t$ as shortening $s_2$ until one of its endpoints becomes $t$, and then extending $s_1$ until it hits another segment or $B$. Parts (a) and (b) of Figure \ref{c} depict a rotation of the intersection $t_1$. A rotation is \textit{non-valid} if an endpoint of another segment is contained in $s_2$ before the rotation but not after it. Figure \ref{c} depicts one valid rotation and one non-valid rotation.

\begin{figure}[h]
    \centering
    \includegraphics[scale=.8]{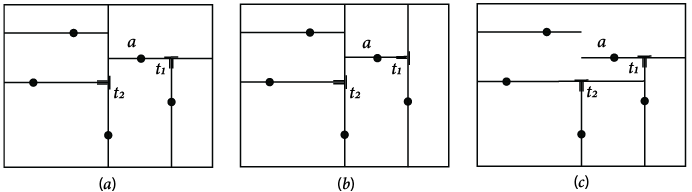}
    \caption{\small \sf Valid and non-valid rotations. (a) A rectangulation $G$. (b) The valid rotation of $t_1$. (c) The non-valid rotation of $t_2$.}
    \label{c}
\end{figure}

In the following, we sometimes refer to left and right endpoints of a segment. This notation applies only to horizontal segments. One can handle vertical segments by symmetrically considering top and bottom endpoints. 

Every segment of degree $i>2$ can be shortened by applying a valid rotation. Indeed, let $s=(a, G)$ be a segment of degree $i>2$. Let $t$ be the rightmost intersection on $s$ that is not the right endpoint of $s$. Without loss of generality, we assume that $t$ is to the right of $p$ (otherwise, we symmetrically define $t$ as the leftmost intersection on $s$ that is not the left endpoint). By definition, the rotation of $t$ is valid and shortens $s$, as asserted. For an example of this, see the segment of $a$ in Figure \ref{c}.

Let $s=(a, G)$ be a segment of degree $i>2$. By the preceding paragraph, we can repeatedly shorten $s$ through valid rotations, until we obtain a segment $s'=(a, G')$ of degree 2. We say that $s'$ is the degree 2 segment obtained by \emph{trimming} $s$. For example, the valid rotation of $t_1$ in Figure \ref{c} trims $(a, G)$ down to a degree 2 segment.

\section{Proof of Theorem \ref{mainthm}}\label{sec3}

The following lemma exposes a connection between $\hat{d}_2$ and $\rc(n)$. A somewhat similar argument appears in Theorem 4.3.1 of \cite{ack}.

\begin{lemma} \label{le:rcVSavgDeg}
For $n\ge2$, assume that there exists $\delta_n>0$ such that for every set $\pts$ of $n$ points in general position we have that $\hat{d}_2(\pts)\ge\delta_nn$. Then, $\rc(n)\le \frac{2}{\delta_n}\rc(n-1)$.
\end{lemma}

\begin{proof} Let $\pts$ be a set that maximizes $\rc(\pts)$ among all sets of $n$ points in the plane. That is, $\rc(\pts)=\rc(n)$. Rectangulations of $\pts$ can be generated by choosing a point $q\in \pts$ and a rectangulation $G$ of $\pts\setminus\lbrace q \rbrace$, inserting $q$ into $G$, and adding a degree 2 segment that contains $q$. In this manner, a rectangulation $G$ of $\pts$ can be obtained in exactly $d_2(G)$ ways. In particular, if $d_2(G)=0$ then $G$ cannot be obtained at all in this fashion. Therefore, it follows that
\begin{equation}
    \hat{d}_2 \cdot \rc(\pts)=\sum_{G\in R(\pts)}d_2(G)= 2\sum_{q\in \pts}\rc(S\setminus\lbrace q\rbrace ).
\end{equation}

Here, the leftmost expression  equals $\hat{d}_2\cdot \rc(n)$, and the rightmost expression is at most $2n\cdot \rc(n-1)$. Recalling the assumption $\hat{d}_2\ge\delta_nn$, we have that \[\rc(n)=\rc(\pts)\le\frac{2N}{\hat{d}_2}\cdot \rc(n-1)\le\frac{2}{\delta_n}\cdot \rc(n-1).\]
\end{proof}
The following lemma is our first lower bound for $\hat{d}_2(n)$.

\begin{lemma} \label{le:weakDeg2}
For every set $\pts$ of $n$ points in general position within a rectangle, $\hat{d}_2(\pts)\ge\frac{n}{9}$.
\end{lemma}

\begin{proof}We use a charging scheme where, initially, every segment of degree $i$ is given $5-i$ units of charge. The sum of the charges of the segments in a rectangulation $G\in R(\pts)$ is
\begin{equation}
    \sum_{i\ge 2} (5-i)d_i(G)=5\sum_{i\ge 2}d_i(G)-\sum_{i\ge 2} i\cdot d_i(G)=5n-\sum_{i\ge 2} i\cdot d_i(G). 
\end{equation}
 
Every intersection in $G$ is the endpoint of one segment, and increases the degrees of at most two segments by 1. Since the segments of $G$ have $2n$ endpoints, the total sum of the degrees in $G$ is at most $4n$. In other words, 
    \[\sum_{i\le 2}i\cdot d_i(G)\le 4n.\]
Combining this with (3), we obtain that the sum of the charges of the segments of $G$ is at least $5n-\sum_i id_i(G)\ge n$. This tells us that, on average, every segment in $G$ has a charge of at least $1$.

For a segment $(a, G)$, we denote by $c(a, G)$ the charge of that segment. Let 
\[C=\sum_{G\in R(\pts)}\sum_{a\in \pts}c(a, G).\] 
In other words, $C$ is the total charge taken over all segments of all rectangulations. By the preceding paragraph, we have that $C\ge n\cdot \rc(\pts)$.

We now move the entire charge to segments of degree 2, as follows. Let $s=(a, G)$ be a segment of degree $i>2$ and let $s'=(a, G')$ be the segment obtained by trimming $s$. We move the entire charge of $s$ to $s'$. Note that this process does not change the total amount of charge. Thus, we still have that $C\ge n\cdot \rc(\pts)$. Let $t$ denote the maximum charge that any degree 2 segment has after the move. Then, $C\le t\cdot\sum_{G\in R(\pts)} d_2(G)$. Combining the bounds for $C$, we have that $n\cdot \rc(\pts)\le C\le t\cdot\sum_{G\in R(\pts)} d_2(G)$. Rearranging this equation leads to 
\begin{equation}\label{4}
\hat{d}_2(\pts)\ge n/t.
\end{equation}

\begin{figure}
    \centering
    \includegraphics[scale=0.76]{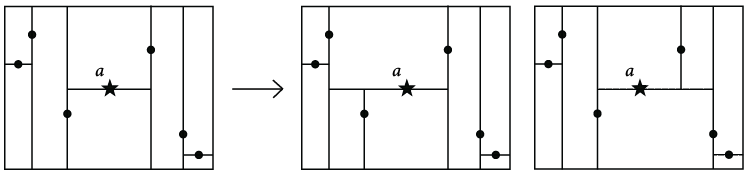}
    \caption{\small\sf The segment of point $a$ receives charge from two segments of degree 3: one obtained by an extension of the segment to the left and the other by an extension to the right.}
    \label{f}
\end{figure}

Any segment $s=(a, G)$ receives charge from at most two segments of degree 3. One of these degree 3 segments is obtained by rotating the left endpoint of $s$. The other is obtained by rotating the right endpoint of $s$. This is demonstrated in Figure \ref{f}. Similarly, $s$ can receive charge from at most three segments of degree 4: one by doing two rotations to the right, another by two rotations to the left, and the third by one rotation to the right and one to the left. Note that $s$ does not get positive charge from any other segment. Thus, the maximum charge $s$ can have is $3\cdot1+2\cdot 2+1\cdot 3=10$.

For $s$ to receive such a charge of 10, it must be extendable at least twice to the right and at least twice to the left.  In this case, $s$ also receives a charge of -1 from at least one segment of degree 6. We conclude that the maximum charge $s$ can have is 9. By repeating the above process for every degree 2 segment, we obtain that $t=9$. Combining this with (\ref{4}) implies that $\hat{d}_2\ge n/9$.
\end{proof}

Combining Lemmas \ref{le:rcVSavgDeg} and \ref{le:weakDeg2} with an induction on $n$ immediately implies the following result.

\begin{theorem} \label{th:Weak}
$\rc(n)\le 18^n$.
\end{theorem}

Theorem \ref{th:Weak} already improves the result from \cite{ack} by a factor of $n^4$. However, we are interested in an exponential improvement. We obtain such an improvement by introducing a more involved charging scheme.

\begin{lemma} \label{le:strongDeg2}
For every set $\pts$ of $n$ points in general position within a rectangle, $\hat{d_2}(\pts)\geq n/(8+5/12)$.
\end{lemma}
\begin{proof}
We begin by applying the same charging scheme as in the proof of Lemma \ref{le:weakDeg2}. 
We then add an additional step of moving charge between degree 2 segments, as described below.
After this step, every degree 2 segment has a charge of at most $8+5/12$.

\begin{figure}[htp]     \centering
    \includegraphics[width=10cm]{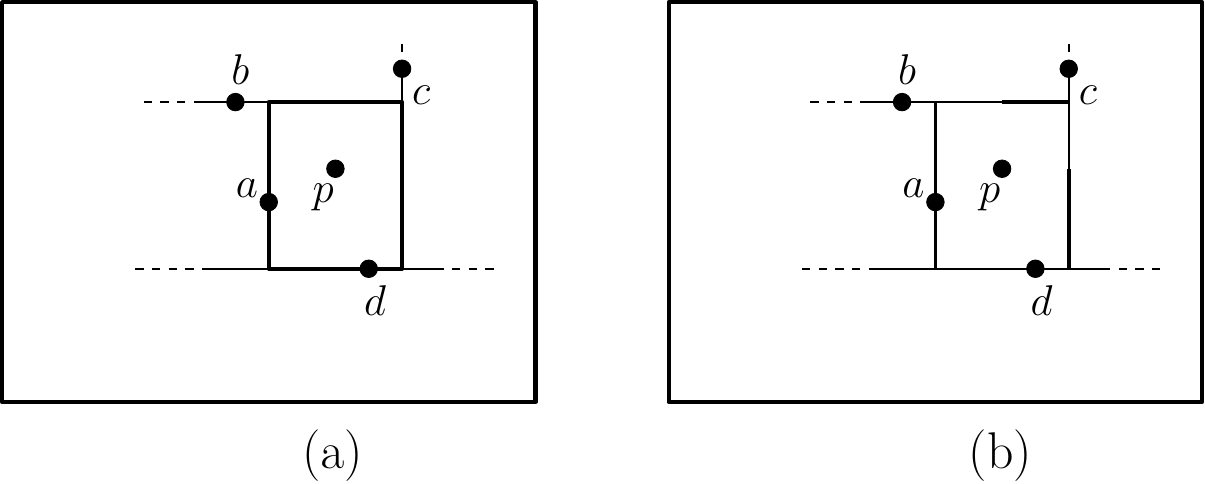}
    \caption{A point $p$ and its four boundary segments. (a) The internal rectangle is in bold. (b) The critical subsegments are bolds}\label{fi:InternalRect}
\end{figure}

Consider a point $p$ and a rectangulation $G$ of $\pts\setminus \{p\}$.
Let $(a,G), (b,G), (c,G),$ and $(d,G)$ be the segments directly to the left, above, to the right, and below $p$, respectively. 
We refer two these four segments as the \emph{boundary segments} of $p$.
See Figure \ref{fi:InternalRect}(a).
A boundary segment might be part of the external rectangle that contains $\pts$, and thus contain no point of $\pts$.
The \emph{internal rectangle} is the rectangle that contains $p$ and is formed by parts of the four boundary segments.
In Figure \ref{fi:InternalRect}(a), the internal rectangle is in bold.

Consider a boundary segment $(v,G)$ that has exactly one of its endpoints on another boundary segment. 
In Figure \ref{fi:InternalRect}, the segments $(b,G)$ and $(c,G)$ have this property, but not $(a,G)$ and $(d,G)$. 
Let $\ell$ be the axis-parallel line that passes through $p$ and intersects $(v,G)$. 
The \emph{critical subsegment} of $(v,G)$ is the open segment in $(v,G)$ between the intersection point with $\ell$ and the endpoint of $(v,G)$ that is on another boundary segment. 
In Figure \ref{fi:InternalRect}(b), the two critical subsegments are bold.
The motivation behind the definition of a critical subsegment will become clear in the following paragraphs.

\begin{figure}[htp]     \centering
    \includegraphics[width=10cm]{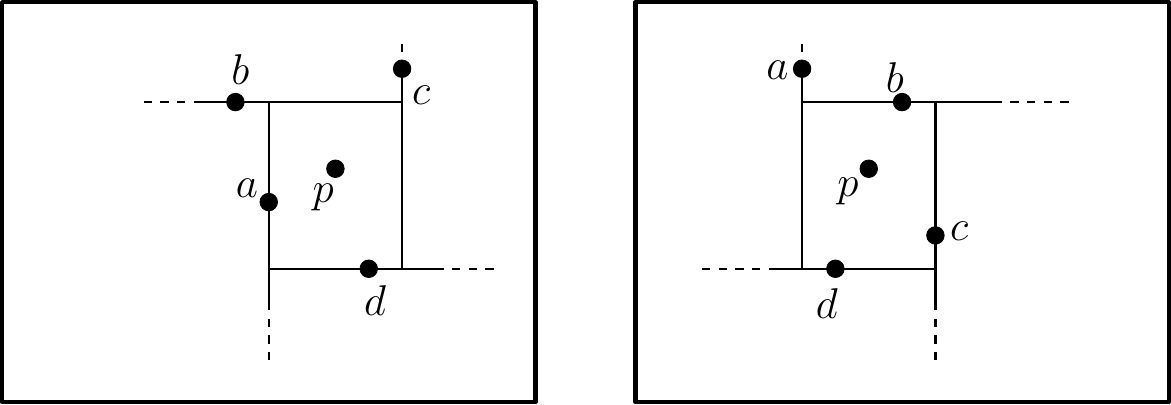}
    \caption{Spiral configurations with opposite orientations.}\label{fi:Spirals}
\end{figure}

We say that $p$ is in a \emph{spiral configuration} if every boundary segment $(v,G)$ satisfies the following:
\begin{itemize}[noitemsep,topsep=1pt]
	\item The boundary segment is not part of the external rectangle.    
    \item Exactly one endpoint of $(v,g)$ is contained in the interior of another boundary segment.
    \item  The critical subsegment of $(v,G)$ does not contain $v$ and does not contain intersection points.  
\end{itemize}
Figure \ref{fi:Spirals} depicts two spiral configurations.

There are two ways of adding in $G$ a degree 2 segment that contains $p$: adding a horizontal segment and adding a vertical segment.
We define the charge of $p$ as the sum of the charges of both degree 2 segments.
To prove the theorem, we move charge across degree 2 segments until every such $p$ has a charge of at most $16+2/3$.
We then equally divide this charge among the two corresponding degree 2 segments, to obtain that every degree 2 segment has a charge of at most $8+1/3$.

\begin{figure}[htp]     \centering
    \includegraphics[width=10cm]{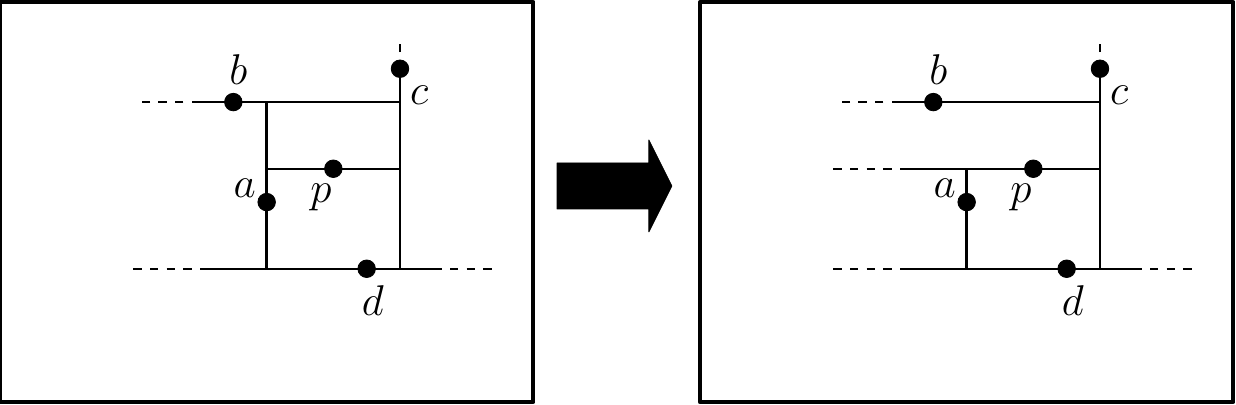}
    \caption{A case where $p$ is extendable beyond $(a,G)$.}\label{fi:Extendable}
\end{figure}

Consider a boundary segment $(v,G)$. 
This segment contains one endpoint $t$ of a degree 2 segment of $p$.
We say that $p$ is \emph{extendable} beyond $(v,G)$ if we can rotate $t$.
That is, $p$ is extendable beyond a boundary segment if a degree 2 segment of $p$ can be extended in that direction.
See Figure \ref{fi:Extendable}.

A main observation: \emph{We can extend $p$ in all four directions if and only if $p$ is in a spiral configuration}. 
Indeed, $p$ is not extendable beyond a boundary segment that has no endpoint on another boundary segment.
If a boundary segment has both of its endpoints on other boundary segments, then there exists another boundary segment with no endpoints on other boundary segments.
Thus, for $p$ to be extendable in all directions, all four boundary segments must have exactly one point on another boundary segment.
If the critical subsegment of a boundary segment $(v,G)$ contains $v$ or an intersection point, then $p$ cannot be extended beyond $(v,G)$. 
There is no other way for $p$ to not be extendable in some direction.

We first consider the case where $p$ is not in a spiral configuration. 
In this case, $p$ is not extendable in at least one direction. 
A degree 2 segment that is not extendable in one direction receives charge from itself, from at most one degree 3 segment, and from at most one degree 4 segment. 
Thus, such a degree 2 segment has at most $3+2+1=6$ charge. 
The proof of Lemma \ref{le:weakDeg2} shows that every degree 2 segment receives a charge of at most 9.
We conclude that, in this case, $p$ has a charge of at most $6+9=15$.

\begin{figure}[htp]     \centering
    \includegraphics[width=10cm]{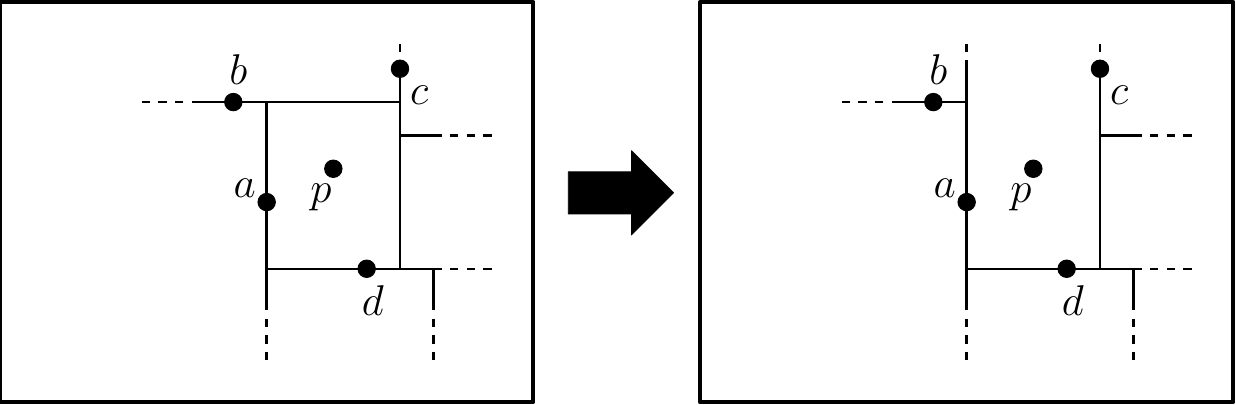}
    \caption{Applying procedure 1, we rotate the intersection of $(a,G)$ and $(b,G)$. }\label{fi:Rotation}
\end{figure}

\parag{Procedure 1.}
It remains to consider the case where $p$ is in a spiral configuration. 
We first assume that we can rotate at least one of the corners of the internal rectangle of $p$.
In this case, we arbitrarily choose a corner that can be rotated and rotate it. 
In Figure \ref{fi:Rotation}, we can rotate the two corners that involve $(a,G)$ and choose to rotate the intersection of $(a,G)$ and $(b,G)$.
We denote the resulting rectangulation as $G'$. 
In $G'$, we cannot extend $p$ beyond the boundary segment that was made longer by the rotation.
After the rotation, this segment has one intersection point in its critical subsegment.
In Figure \ref{fi:Rotation}, this is the segment $(a,G)$.
Thus, $p$ has a charge of at most 15 in $G'$.
We move 1.5 charge from $p$ in $G$ to $p$ in $G'$.
That is, $p$ now has at most 16.5 charge both in $G$ and in $G'$.
We refer to the above as \emph{procedure 1}. 

After applying procedure 1, we have the following situation around $p$:
\begin{itemize}[noitemsep,topsep=1pt]
    \item The boundary segment that was made longer by the rotation has exactly one intersection in its critical subsegment. This segment still has exactly one endpoint contained on another boundary segment.
    \item We might be able extend $p$ beyond the boundary segment $(v,G)$ that contains the endpoint of the segment in the preceding bullet. However, it is also possible for $(v,G)$ to contain intersections or $v$ on its critical subsegment. It is also possible that $(v,G)$ has no endpoint on another boundary segment. 
    \item We can extend $p$ beyond the other two boundary segments. 
\end{itemize}

Every point may be charged at most once by procedure 1. 
Indeed, if a point $p$ is charged in this way in a rectangulation $G'$, then $p$ is not extendable in at most two adjacent directions in $G'$. 
Out of the corresponding boundary segments, the one that has an endpoint on the other must have exactly one intersection point on its critical subsegment. 
We obtain the charging rectangulation $G$ by rotating this intersection point. 

\begin{figure}[htp]     \centering
    \includegraphics[width=11cm]{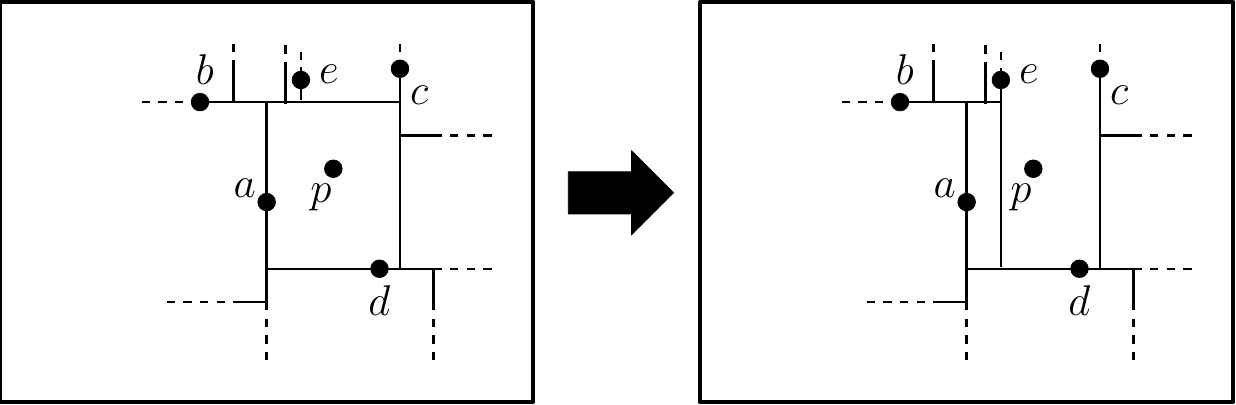}
    \caption{Applying procedure 2, we rotate the intersection of $(b,G)$ and $(e,G)$. }\label{fi:Procedure2}
\end{figure}

\parag{Procedure 2.} It remains to consider the case where $p$ is in a spiral configuration and all corners of the internal rectangle cannot be rotated. 
In this case, intersections on the interior of the boundary segments prevent us from rotating the corners.
Since $p$ is in a spiral configuration, these intersections are not on critical subsegments. 
In the left side of Figure \ref{fi:Procedure2}, two intersections prevent us from rotating the intersection of $(a,G)$ and $(b,G)$.
In the same figure, one intersection prevents us from rotating the intersection of $(c,G)$ and $(d,G)$.

Consider a boundary segment $(v,G)$ that contains an intersection point $t$, such that $p$ is closer to $t$ than it is to $v$.
Then the intersection on $(v,G)$ that is closest to $p$ can be rotated.  
In Figure \ref{fi:Procedure2}, we can rotate the intersection of $(b,G)$ and $(e,G)$.
If such an intersection exists, then we arbitrarily rotate one such intersection and denote the resulting rectangulation as $G'$.
We note that $p$ is not in a spiral configuration in $G'$. 
Indeed, in $G'$, the boundary segment on the opposite side of the internal rectangle has no endpoint on another boundary segment.
In Figure \ref{fi:Procedure2}, this is the segment $(d,G')$.
We move 1.5 charge from $p$ in $G$ to $p$ in $G'$.
That is, $p$ now has at most 16.5 charge both in $G$ and in $G'$.
We refer to the above as \emph{procedure 2}. 

After applying procedure 2, we have the following situation around $p$:
\begin{itemize}[noitemsep,topsep=1pt]
    \item One boundary segment has no endpoint on another boundary segment. The critical subsegment of this boundary segment is empty.
    \item We can extend $p$ beyond the two boundary segments that are adjacent to the one in the preceding bullet.
    \item The fourth boundary segment, which is opposite to the first segment, may have any number of issues or none. 
\end{itemize}

No point can be charged by both procedure 1 and procedure 2. 
Indeed, when $p$ is not extendable in exactly one direction, each procedure requires a different cause for this (no endpoint on another boundary segment or an intersection on the critical subsegment). 
When $p$ is not extendable in two directions, each procedure requires a different relation between these two directions (opposite directions or adjacent directions). 

\begin{figure}[htp]     \centering
    \includegraphics[width=14cm]{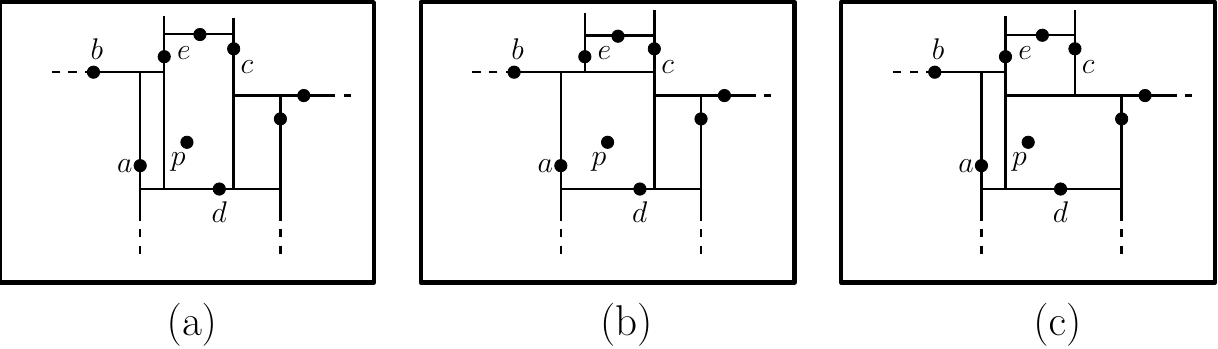}
    \caption{A point that is extendable in three directions cannot be charged twice by procedure 2. }\label{fi:Procedure2twice}
\end{figure}

It is tempting to revise the configuration on the right-hand side of Figure \ref{fi:Procedure2} to be symmetric on both sides. 
That is, to have $p$ not extendable only in one direction while getting charged twice by procedure 2.
This case is depicted in Figure \ref{fi:Procedure2twice}(a), where the procedure 2 charge is supposed to come from the points in Figures \ref{fi:Procedure2twice}(b) and \ref{fi:Procedure2twice}(c).
However, since $p$ is not in spiral configuration in Figure \ref{fi:Procedure2twice}(c), we do not apply procedure 2 in this case.
Figure \ref{fi:Procedure2twice}(a) is not really symmetric since $d$ is to the right of $p$, so $p$ is not extendable below in Figure \ref{fi:Procedure2twice}(c).
We conclude that a point that is extendable in three directions cannot be charged twice by procedure 2. 

By imitating the preceding paragraph, we can show that a point that is extendable in two opposite directions cannot be charged twice by procedure 2. 
However, it is easier to address this case in a different way.
If such a configuration exists, then one of the degree 2 segments of $p$ is not extendable in either direction.
This degree 2 segment has a charge of 3, so $p$ has a charge of at most $3+9=12$.
After receiving charge by two applications of procedure 2, this $p$ has a charge of at most 15.

\begin{figure}[htp]     \centering
    \includegraphics[width=4cm]{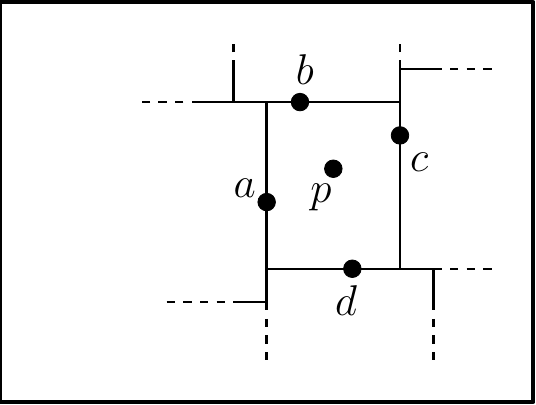}
    \caption{The case where we cannot apply procedures 1 and 2. }\label{fi:LastCase}
\end{figure}

\parag{Procedure 3.} 
It remains to consider the case where $p$ is in a spiral configuration, no corner of the internal rectangle can be rotated, and procedure 2 cannot be applied.
An example is depicted in Figure \ref{fi:LastCase}.
In this case, the points of all four boundary segments are on the boundary of the internal rectangle.  
Every boundary segment contains an endpoint of another segment in the part that does not participate in the internal rectangle.

\begin{figure}[htp]     \centering
    \includegraphics[width=11cm]{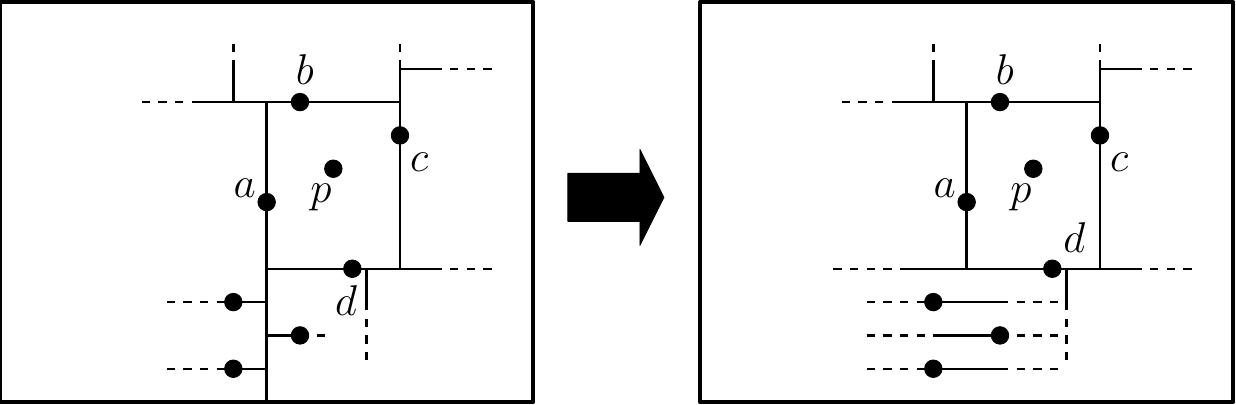}
    \caption{Applying procedure 3, we trim $(a,G)$. }\label{fi:Procedure3}
\end{figure}

The segment $(a,G)$ contains an endpoint of another boundary segment and the endpoint of at least one non-boundary segment. 
Thus, the degree of $(a,G)$ is at least 4.
We trim this segment, as explained in Section \ref{sec2}.
See Figure \ref{fi:Procedure3}.
We then remove the resulting degree 2 segment of $a$, add a vertical degree 2 segment for $p$, and denote the resulting rectangulation as $G'$.
See Figure \ref{fi:Procedure3b}(a).
We move 1 charge from $p$ in $G$ to $a$ in $G'$. 
We refer to the above as \emph{procedure 3}. 

After applying procedure 3, we have the following situation around $a$:
\begin{itemize}[noitemsep,topsep=1pt]
    \item The boundary segment to the right is of degree 2, so $a$ is extendable to the right. 
    \item We can extend $a$ either to the top or to the bottom. The boundary segment at the opposite side has an intersection on its critical subsegment. It may also have no endpoint on another boundary segment.
    \item The segment to the left may have any number of issues, or none. 
\end{itemize}

No point receives charge from procedure 2 and procedure 3.
Indeed, for $a$ to receive charge from procedure 2, there needs to exist a boundary segment $(v,G)$ with no endpoints on other boundary segments.
Also, $a$ needs to be extendable beyond the two boundary segments that are adjacent to $(v,G)$. 
This cannot happen after applying procedure 3.

For $a$ to also receive charge from procedure 1, there needs to be a boundary segment $(v,G')$ with exactly one intersection point on its critical subsegment.
If $a$ is extendable in the other three directions, then $(v,G')$ is either above or below $a$.
Rotating the intersection point in the critical subsegment of $(v,G')$ must put $a$ in a spiral configuration. 
This is impossible, since after the rotation, the boundary segment to the right of $a$ still has two endpoints on other boundary segments.
Thus, if $a$ is charged by procedure 1 and procedure 3, then it is not extendable in two directions.
If $a$ is not extendable in two directions, then it has a charge of at most 12 before applying these procedures.

\begin{figure}[htp]     \centering
    \includegraphics[width=11cm]{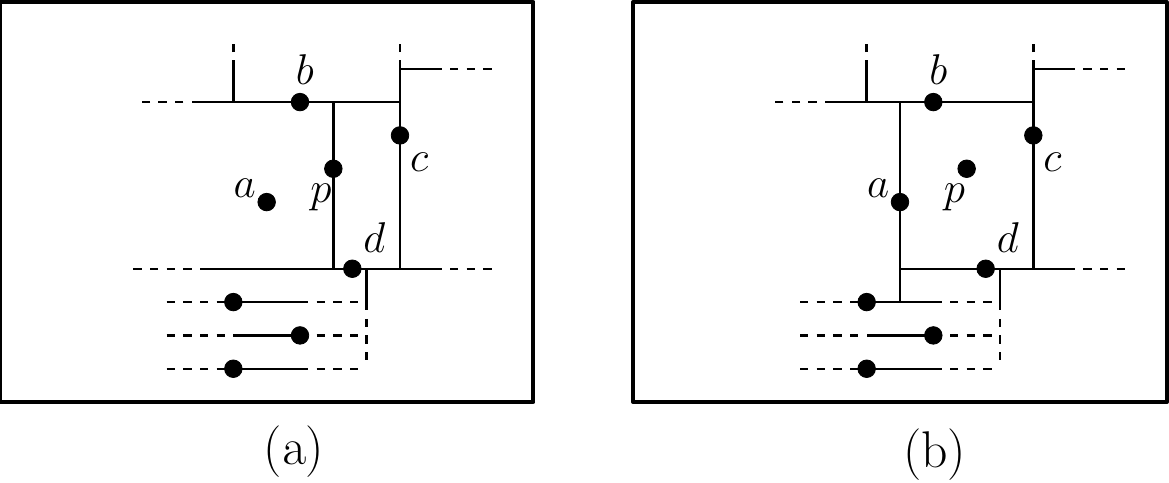}
    \caption{(a) The situation after applying procedure 3 in Figure \ref{fi:Procedure3}. (b) When extending the vertical segment of $a$ once we obtain a corner that could be rotated. }\label{fi:Procedure3b}
\end{figure}

\parag{Completing the analysis.}
If $p$ does not receive charge from procedure 3, then it has a charge of at most $16.5$ (see procedures 1 and 2 above). 
We thus consider a point $a$ that is charged as described in procedure 3. 
Let $k$ be the number of times that the vertical 2 segment of $a$ can be extended below (or above, if the spiral has the opposite orientation than the one in Figure \ref{fi:Procedure3}).
Equivalently, the degree 2 vertical segment of $a$ in $G'$ is extendable 0 times in one direction and $k$ times in the other. 
In Figure \ref{fi:Procedure3}(a), we have that $k=4$.
Before applying any procedures, this degree 2 segment has a charge of 
\[ 3+2+\cdots + (3-k) = \frac{(k+1)(6-k)}{2} = 3+\frac{5k-k^2}{2}.\]
Since the horizontal segment of $a$ receives at most 9 charge, the total charge of $a$ is at most $12+\frac{5k-k^2}{2}$. 

After extending the vertical degree 2 segment of $a$ once and removing the segment of $p$, the resulting point $p$ is handled according to procedure 1.
In particular, one of the left corners of the internal rectangle could be rotated.
See Figure \ref{fi:Procedure3b}.
Each of the following $k-1$ extensions of the vertical segment of $a$ may lead to a point $p$ that uses procedure 3 to charge $a$.
Since each such application of procedure 3 increases the charge of $a$ by 1, the charge of $a$ is now at most
\[ 12+\frac{5k-k^2}{2} +(k-1) = 11+\frac{7k-k^2}{2}. \] 

We recall that $a$ might receive an additional 1.5 charge from procedure 1. 
However, in that case the horizontal degree 2 segment of $a$ has a charge of at most 6, so the total charge goes down by 1.5.
Thus, the maximum charge $a$ may have remains $11+\frac{7k-k^2}{2}$.
This expression is maximized when $k$ is either 3 or 4, and is then 17.
This completes the proof that $\hat{d_2}(\pts)\geq 17/2 = 8.5$.

\parag{A final improvement.}
We recall the application of procedure 1 from two paragraphs above. 
This procedure ends with two instances of the point $p$ that have a charge of at most $16.5$.
These two instances of $p$ do not receive additional charge from procedure 3, since the segment to their right is not of degree 2.  
We also recall that $a$ receives a charge of at most 17, and so do the other $k-1\le 3$ instances of $p$ that move charge to $a$.
Thus, we have at most six points with a total charge of at most $4\cdot 17 + 2\cdot 16.5 = 101$.
We evenly redistribute this charge by giving each of the six points a charge of at most $101/6 = 16+5/6$. 

A careful reader may comment that, as $k$ becomes larger, we have more instances of $p$ with a charge of at most 17.
Thus, equally spreading the charge among $k+2$ points may lead to a worse bound when $k>4$. 
However, it could be easily verified that, each time we increase $k$ by 1, the original charge of $a$ decreases by at least 1. 
This decrease in the charge is more than enough to compensate for the extra point. 
The worst case is obtained when $k=4$.
\end{proof}

Combining Lemmas \ref{le:rcVSavgDeg} and \ref{le:strongDeg2} with an induction on $n$ immediately implies Theorem \ref{mainthm}.

\parag{Acknowledgements.}
We would like to thank Professor Adam Sheffer for his mentorship and for the suggestion of this topic.

\end{document}